\documentclass[a4paper,12pt,final]{amsart}
\usepackage{times,a4wide,mathrsfs,amssymb}

\newcommand{\C}{\mathbb{C}}
\newcommand{\ZZ}{\mathbb{Z}}

\newcommand{\QQ}{\mathbb{Q}}
\newcommand{\NN}{\mathbb{N}}
\newcommand{\PP}{\mathbb{P}}
\newcommand{\LL}{\mathbb{L}}

\newcommand{\OO}{\mathcal O}

\newcommand{\DD}{\mathcal D}
\newcommand{\XX}{\mathcal X}
\newcommand{\YY}{\mathcal Y}

\newcommand{\Z}{\mathcal Z}

\newcommand{\MM}{\mathcal M}

\newcommand{\gr}{\hbox{Gr}}
\newcommand{\wt}{\widetilde}
\newcommand{\ima}{\hbox{Im}}
\newcommand{\rom}{\romannumeral}

\newcommand\undermat[2]{
  \makebox[0pt][l]{$\smash{\underbrace{\phantom{
    \begin{matrix}#2\end{matrix}}}_{\text{$#1$}}}$}#2}

\newtheorem{theorem}{Theorem}[section]

\newtheorem{lemma}[theorem]{Lemma}
\newtheorem{corollary}[theorem]{Corollary}
\newtheorem{proposition}[theorem]{Proposition}
\newtheorem{conjecture}[theorem]{Conjecture}
\newtheorem{remark}[theorem]{Remark}
\newtheorem{definition}[theorem]{Definition}
\newtheorem{convention}{Conventions}

\newtheorem{notation}[theorem]{Notation}

\newtheorem{nonumbering}{Theorem}

\newtheorem{nonumberingc}{Corollary}

\newtheorem{nonumberingt}{Acknowledgements}

\begin{document}
\author[Robert Laterveer]
{Robert Laterveer}

\address{Institut de Recherche Math\'ematique Avanc\'ee,
CNRS -- Universit\'e 
de Strasbourg,\
7 Rue Ren\'e Des\-car\-tes, 67084 Strasbourg CEDEX,
FRANCE.}
\email{robert.laterveer@math.unistra.fr}

\title{A family of cubic fourfolds with finite--dimensional motive}

\begin{abstract} We prove that cubic fourfolds in a certain $10$--dimensional family have finite--dimensional motive. The proof is based on the van Geemen--Izadi construction of an algebraic Kuga--Satake correspondence for these cubic fourfolds, combined with Voisin's method of ``spread''. Some consequences are given.
 \end{abstract}



\maketitle

\section{Introduction}

The notion of finite--dimensional motive, developed independently by Kimura and O'Sullivan \cite{Kim}, \cite{An}, \cite{MNP}, \cite{J4}, \cite{Iv} has given considerable new impetus to the study of algebraic cycles. To give but one example: thanks to this notion, we now know the Bloch conjecture is true for surfaces of geometric genus zero that are rationally dominated by a product of curves \cite{Kim}. It thus seems worthwhile to find concrete examples of varieties that have finite--dimensional motive, this being (at present) one of the sole means of arriving at a satisfactory understanding of Chow groups. 

The object of the present note is to add to the list of examples of varieties with finite--dimensional motive, by considering cubic fourfolds over $\C$. There is one famous cubic fourfold with finite--dimensional motive: the Fermat cubic
  \[  x_0^3+x_1^3+x_2^3+x_3^3+x_4^3+x_5^3=0\ .\]
  The Fermat cubic has finite--dimensional motive because it is rationally dominated by a product of (Fermat) curves, and the indeterminacy locus is again of Fermat type \cite{Shi}.
  
  The main result of this note proves finite--dimensionality for a $10$--dimensional family of cubic fourfolds containing the Fermat cubic:
 
 \begin{nonumbering}[=theorem \ref{main}] 
Let $X\subset\PP^5(\C)$ be a smooth cubic fourfold, defined by an equation
  \[  f(x_0,\ldots,x_4)+x_5^3=0\ ,\]
  where $f(x_0,\ldots,x_4)$ defines a smooth cubic threefold. Then $X$ has finite--dimensional motive.
\end{nonumbering}

Unlike the Fermat cubic, the cubics as in theorem \ref{main} are {\em not\/} obviously dominated by a product of curves, so we need some more indirect reasoning. In a nutshell, the idea of the proof of theorem \ref{main} is as follows: thanks to the work of van Geemen--Izadi \cite{GI}, there exists a Kuga--Satake correspondence for these special cubic fourfolds. This implies that the homological motive of $X$ is a direct summand of the motive of an abelian variety. Then, considering the family of all cubic fourfolds as in theorem \ref{main} and using the machinery developed by Voisin \cite{V0}, \cite{V1} and L. Fu \cite{LFu}, we can upgrade this relation to rational equivalence and prove the Chow motive of $X$ is a direct summand of the motive of an abelian variety.

We present some consequences of finite--dimensionality. One consequence is the verification of (a weak form of) the Bloch conjecture for these special cubic fourfolds:

\begin{nonumberingc}[=corollary \ref{corbloch}] Let $X$ be a cubic fourfold as in theorem \ref{main}. Let $\Gamma\in A^4(X\times X)$ be a correspondence such that
  \[  \Gamma_\ast\colon\ \ H^{3,1}(X)\ \to\ H^{3,1}(X)\]
  is the identity. Then
  \[ \Gamma_\ast\colon\ \ A^3_{hom}(X)\ \to\ A^3_{hom}(X)\]
  is an isomorphism.
  \end{nonumberingc}

Another consequence (proposition \ref{corsmash}) concerns Voevodsky's smash--nilpotence conjecture for products $X_1\times X_2$, where $X_1, X_2$ are cubic fourfolds as in theorem \ref{main}.



\vskip0.8cm

\begin{convention} In this note, the word {\sl variety\/} will refer to a reduced irreducible scheme of finite type over $\C$. A {\sl subvariety\/} is a (possibly reducible) reduced subscheme which is equidimensional. 

{\bf All Chow groups are with rational coefficients}: we will denote by $A_jX$ the Chow group of $j$--dimensional cycles on $X$ with $\QQ$--coefficients; for $X$ smooth of dimension $n$ the notations $A_jX$ and $A^{n-j}X$ will be used interchangeably. 

The notations 
$A^j_{hom}(X)$ and $A^j_{AJ}(X)$ will be used to indicate the subgroups of 
homologically, resp. Abel--Jacobi trivial cycles.
For a morphism $f\colon X\to Y$, we will write $\Gamma_f\in A_\ast(X\times Y)$ for the graph of $f$.
The category of Chow motives (i.e., pure motives with respect to rational equivalence as in \cite{Sc}, \cite{MNP}) will be denoted $\MM_{\rm rat}$.

To avoid heavy notation, if $\tau\colon Y\to X$ is a closed immersion and $a\in A_i(Y)$, we will frequently write $a\in A_i(X)$ to indicate the proper push--forward $\tau_\ast(a)$. Likewise, for any inclusion $Y\subset X$ and $b\in A^j(X)$ we will often write
  \[   b\vert_{Y}\ \ \in A^j(Y)\]
  to indicate the cycle class $\tau^\ast(b)$.


We will write $H^j(X)$ 
and $H_j(X)$ 
to indicate singular cohomology $H^j(X,\QQ)$,
resp. singular homology $H_j(X,\QQ)$.

\end{convention}

\vskip0.8cm

\section{Preliminaries}

\subsection{Finite--dimensional motives}

We refer to \cite{K}, \cite{An}, \cite{Iv}, \cite{J4}, \cite{MNP} for the definition of finite--dimensional motive. 
An essential property of varieties with finite--dimensional motive is embodied by the nilpotence theorem:

\begin{theorem}[Kimura \cite{K}]\label{nilp} Let $X$ be a smooth projective variety of dimension $n$ with finite--dimensional motive. Let $\Gamma\in A^n(X\times X)_{\QQ}$ be a correspondence which is numerically trivial. Then there is $N\in\NN$ such that
     \[ \Gamma^{\circ N}=0\ \ \ \ \in A^n(X\times X)_{}\ .\]
\end{theorem}

 Actually, the nilpotence property (for all powers of $X$) could serve as an alternative definition of finite--dimensional motive, as shown by a result of Jannsen \cite[Corollary 3.9]{J4}.
Conjecturally, any variety has finite--dimensional motive \cite{K}. We are still far from knowing this, but at least there are quite a few non--trivial examples:
 
\begin{remark} 
The following varieties have finite--dimensional motive: abelian varieties, varieties dominated by products of curves \cite{K}, $K3$ surfaces with Picard number $19$ or $20$ \cite{P}, surfaces not of general type with vanishing geometric genus \cite[Theorem 2.11]{GP}, Godeaux surfaces \cite{GP}, Catanese and Barlow surfaces \cite{V8}, certain surfaces of general type with $p_g=0$ \cite{PW}, Hilbert schemes of surfaces known to have finite--dimensional motive \cite{CM}, generalized Kummer varieties \cite[Remark 2.9(\rom2)]{Xu},
 3--folds with nef tangent bundle \cite{Iy} (an alternative proof is given in \cite[Example 3.16]{V3}), 4--folds with nef tangent bundle \cite{Iy2}, log--homogeneous varieties in the sense of \cite{Br} (this follows from \cite[Theorem 4.4]{Iy2}), certain 3--folds of general type \cite[Section 8]{V5}, varieties of dimension $\le 3$ rationally dominated by products of curves \cite[Example 3.15]{V3}, varieties $X$ with $A^i_{AJ}(X)_{}=0$ for all $i$ \cite[Theorem 4]{V2}, products of varieties with finite--dimensional motive \cite{K}.
\end{remark}

\begin{remark}
It is worth pointing out that all examples of finite-dimensional motives known so far happen to be in the tensor subcategory generated by Chow motives of curves (i.e., they are ``motives of abelian type'' in the sense of \cite{V3}). 
That is, the finite--dimensionality conjecture is still unknown for any motive {\em not\/} generated by curves (on the other hand, there exist many such motives, cf. \cite[7.6]{D}).
\end{remark}

\subsection{Kuga--Satake}
\label{ssKS}

This subsection presents the first main ingredient of this note: the van Geemen--Izadi construction of an algebraic Kuga--Satake correspondence for the cubic fourfolds under consideration.

\begin{theorem}[van Geemen--Izadi \cite{GI}]\label{thGI} Let $X\subset\PP^5$ be a smooth cubic fourfold, defined by an equation
  \[  x_5^3+f(x_0,\ldots,x_4)=0\ ,\]
  where $f(x_0,\ldots,x_4)$ defines a smooth cubic threefold. Let $Z\subset\PP^6$ be the cubic fivefold defined by 
    \[x_6^3+x_5^3+f(x_0,\ldots,x_4)=0\ .\]
There exist an elliptic curve $E$ and a correspondence $\Gamma\in A^5(X\times Z\times E)$ such that
  \[ \Gamma_\ast\colon\ \ H^4(X)_{prim}\ \to\ H^6(Z\times E)\]
  is injective.
 \end{theorem}
 
 \begin{proof} This is \cite[Corollary 5.3]{GI}. This result is based on the facts that (1) the Hodge structure of any smooth cubic fourfold is of $K3$ type (i.e., $H^{4,0}(X)=0$ and $\dim H^{3,1}(X)=1$), and (2) for cubics as in theorem \ref{thGI}, the cyclotomic field $\QQ(\zeta)$ acts on $H^4(X)_{prim}$ (where $\zeta=e^{2\pi i\over 3}$), and so the theory of half twists \cite{vG2} applies. 
 
 We note that \cite[Corollary 5.3]{GI} actually shows more precisely that
   \[  \Gamma_\ast\colon\ \ H^4(X)_{prim}\ \to\ \ima\Bigl( H^5(Z)\otimes H^1(E)\ \to\  H^6(Z\times E)\Bigr)\]
  is injective. Also, as we shall see below (in the proof of theorem \ref{GIfam}), the elliptic curve $E$ is actually a plane cubic of Fermat type $x_0^3+x_1^2+x_2^3=0$.
  \end{proof}

 \begin{corollary}\label{abvar} Let $X$ be as in theorem \ref{thGI}. There exist an abelian variety $A$ (of dimension $22 $) and a correspondence $\Psi\in A^{3}(X\times A)$ such that
   \[ \Psi_\ast\colon\ \ H^4(X)_{prim}\ \to\ H^2(A)\]
  is injective.
  \end{corollary}
  
  \begin{proof} Any smooth cubic fivefold $Z$ has $H^5(Z)=N^2 H^5(Z)$, where $N^\ast$ denotes the geometric coniveau filtration (this follows from the fact that any cubic fivefold $Z$ has $A_0(Z)=A_1(Z)=\QQ$, which is proven in \cite{Lew} or, alternatively, \cite{Ot} or \cite{HI}).

  Now, \cite[Theorem 1]{ACV} furnishes an abelian variety $J$ (of dimension $h^{2,3}(Z)=21$ ) and a correspondence $\Lambda^\prime$ on $J\times Z$ that induces an isomorphism
  \[  (\Lambda^\prime)_\ast\colon\ \ H^1(J)\ \xrightarrow{\cong}\ H^5(Z)\ .\]
  (As noted by the referee, one may avoid recourse to \cite{ACV} here by using the fact that thanks to Collino \cite{Col}, the Abel--Jacobi map induces an isomorphism from the Albanese of the Fano surface of planes in $Z$ to the intermediate Jacobian of $Z$.)
  
  The correspondence $\Lambda^\prime$ induces an isomorphism 
  \[  \Lambda^\prime\colon\ \ h^1(J)\ \xrightarrow{\cong}\ h^5(Z)\ \ \hbox{in}\ \MM_{\rm hom}\ ,\]
  hence there also exists a correspondence $\Lambda$ on $Z\times J$ inducing the inverse isomorphism
  \[ \Lambda \colon\ \ h^5(Z)\ \xrightarrow{\cong}\ h^1(J)\ \ \hbox{in}\ \MM_{\rm hom}\ .\]
  
  The composition
  \[ H^4(X)_{prim}\ \xrightarrow{\Gamma_\ast}\ H^5(Z)\otimes H^1(E)\ \xrightarrow{(\Lambda\times \Delta_E)_\ast}\ H^1(J)\otimes H^1(E)\ \ \subset H^2(J\times E)\]
  has the required properties.
  \end{proof}

\begin{notation}\label{family} Let 
  \[ \XX\ \to\ B \]
  denote the universal family of all smooth cubic fourfolds of type
   \[  x_5^3+f_b(x_0,\ldots,x_4)=0\ ,\]
  where $f_b(x_0,\ldots,x_4)$ defines a smooth cubic threefold. 
  (That is, the parameter space $B$ is a Zariski open in a linear subspace $\bar{B}$ of the complete linear system $\PP H^0(\PP^5,\OO_{\PP^5}(3))$.)
  
  Likewise, let
  \[ \Z\ \to\ B \]
  denote the family of smooth cubic fivefolds of type
  \[ x_6^3+x_5^3+f_b(x_0,\ldots,x_4)=0\ .\]
  
  For $b\in B$, we will write $X_b\subset\PP^5$ and $Z_b\subset\PP^6$ to denote the fibre of $\XX\to B$ (resp. $\Z\to B$) over $b$.
  \end{notation}

  \begin{notation} Let 
    \[ \XX\ \to\ B\ ,\ \ \ \YY\ \to\ B \]
    be two smooth families (i.e., smooth projective morphisms between smooth quasi--projective varieties). A relative correspondence from $\XX$ to $\YY$ is by definition a cycle class in
    \[ A^\ast(\XX\times_B \YY) \ .\]
    As explained in \cite[Section 8.1]{MNP}, using Fulton's refined Gysin homomorphisms \cite{F} one can define the composition of relative correspondences. For a relative correspondence $\Gamma\in A^i(\XX\times_B \YY)$, and a point $b\in B$ the ``restriction to a fibre'' is defined as
    \[  \Gamma\vert_{X_b\times Y_b} := \iota^\ast (\Gamma)\ \ \ \in A^i(X_b\times Y_b)\ ,\]
    where $\iota^\ast$ denotes the refined Gysin homomorphism associated to the lci morphism $\iota\colon b\to B$.
  \end{notation}

  A crucial point in this note is that the Kuga--Satake construction of \cite{GI} can be done family--wise:
  
  \begin{theorem}\label{GIfam} Notation as in \ref{family}. There exists a relative correspondence
     \[ \Gamma_{KS}\ \ \in A^5\bigl( \XX\times_B (\Z\times E)\bigr)\ ,\]
    such that for any $b\in B$, the restriction 
      \[ \Gamma_{KS,b}:=\Gamma_{KS}\vert_{X_b\times Z_b\times E}\ \ \in A^5\bigl(X_b\times (Z_b\times E)\bigr)\]
      has the property that
      \[ (\Gamma_{KS,b})_\ast\colon\ \ H^4(X_b)_{prim}\ \to\ H^6(Z_b\times E)\]
      is injective.
    \end{theorem}
    
    \begin{proof} To prove this, we partially unravel the proof of \cite[Theorem 5.2]{GI} and \cite[Corollary 5.3]{GI}. For a given $b\in B$, let us denote
      \[ V:=H^4(X_b)_{prim}(1)\ \ \]
      (where the Tate twist indicates $V$ is a weight $2$ Hodge structure with $V^{0,2}=1$). 
      The cubic $X_b$ is invariant under the $\ZZ/3\ZZ$ action on $\PP^5$ induced by
      \[  [x_0:\ldots:x_5]\ \mapsto\ [x_0:\ldots :x_4:\zeta x_5]\ ,\]
      where $\zeta=e^{2\pi i\over 3}$.  As such, we have that $V$ is a vector space over $K:=\QQ(\zeta)$. Let $E\subset\PP^2$ denote the degree $3$ Fermat curve. Then $E$ is an elliptic curve with complex multiplication by $K$ (here $K$ acts via mutiplication on the last coordinate), and
       \[ K_{-1/2}\cong H^1(E)\ .\]
    (NB: in the notation of \cite{GI}, the curve $E$ is both $Y_1$ and $A_K$.)   
       The positive half twist $V_{1/2}$ (a Hodge structure of weight $1$)  exists \cite[Example 2.12 and Proposition 2.8]{vG2}, \cite[Theorem 2.6]{GI}. Moreover, there is an equality of Hodge structures of weight $3$
       \[ V_{1/2}(-1)=W:=\Bigl( V\otimes H^1(E)\Bigr)^{<\beta>}\ ,\]
       where $()^{<\beta>}$ denotes the invariant part under a certain automorphism $\beta$ of $X_b\times E$
       \cite[Theorem 3.4 and Lemma 3.7]{GI}. The automorphism $\beta$ is defined as 
              \[  \beta:=((\alpha_4)^\ast,(\alpha_1)^\ast)\colon\ \ \ X_b\times E\ \to\ X_b\times E\ ,\]
          where $\alpha_4$ (resp. $\alpha_1$) is the restriction to $X_b$  (resp. to $E$) of the automorphism of $\PP^5$ given by
           \[    [x_0:\ldots:x_5]\ \mapsto\ [x_0:\ldots :x_4:\zeta x_5]\ \]
           (resp. of the automorphism of $\PP^2$ defined as $[x_0:x_1:x_2]\mapsto [x_0:x_1:\zeta x_2]$).
                      
                   There is a homomorphism  
     \[ \mu_f\colon\ \ V\otimes H^1(E)\ \to\ W\ \ \ \subset H^4(X_b)\otimes H^1(E)\ ,\]
     defined as the projection onto the $\beta$--invariant subspace. The homomorphism $\mu_f$ is induced by a correspondence; what's more, this correspondence comes from a relative correspondence (this is because the automorphism $\beta=(\alpha_4,\alpha_1)$ in \cite[Theorem 3.4]{GI} comes from an automorphism of $\PP^5\times E$, and so for each $X_b$ the homomorphism $\mu_f$ is given by the restriction of a correspondence on $\PP^5\times E\times\PP^5\times E\times B$).
     
     Next, one considers the homomorphism
     \[ \mu_f\otimes\hbox{id}\colon\ \ V\otimes H^1(E)\otimes H^1(E)\ \to\ W\otimes H^1(E)\ \ \ \subset H^4(X_b)\otimes H^1(E)\otimes H^1(E)\ ;\]
     this has the property that
       \[ \ima (\mu_f\otimes\hbox{id})=V_{1/2}(-1)\otimes K_{-1/2}=W\otimes H^1(E)\ .\]
     The domain of $\mu_f\otimes\hbox{id}$ has a certain Hodge substructure $S$ defined as
       \[ S:=\Bigl\{    w\in V\otimes K_{-1/2}\otimes K_{-1/2}\ \vert\ \ ((\alpha_4)^\ast\otimes \zeta\otimes 1)w=w\ ,\ \ (1\otimes \zeta\otimes\zeta)w=w\Bigr\}\ .\]
     One checks that
       \[ S\cong V(-1)\ .\]
       Since $S\subset V_{1/2}(-1)\otimes K_{-1/2}$, the restriction of $\mu_f\otimes\hbox{id}$ to $S$ is injective, and thus
         \[ (\mu_f\otimes\hbox{id})(S)\cong V(-1)\ .\]
    One checks that actually
      \[ S\ \subset\ V\otimes K(-1)\ \subset\ V\otimes K_{-1/2}\otimes K_{-1/2}\ ,\]
      where $K(-1)$ is a trivial weight $2$ rank $2$ Hodge structure. It follows that the (twisted) isomorphism
      \[ \Gamma\colon\ \ V\ \to\ S\cong V(-1) \]
      is induced by a correspondence on $X_b\times X_b\times E\times E$. This correspondence is again the restriction of a relative correspondence (it comes from $\Delta_\XX\times 
      D$, where $D\in A^1(E\times E)$).
      
      Next, the work of Shioda \cite[Theorem 2]{Shi} produces a homomorphism
      \[  Sh\colon\ \ H^4(X_b)\otimes H^1(E)\ \to\ H^5(Z_b)\ .\]
      As $Sh$ comes from a rational map $X_b\times E \dashrightarrow Z_b$, it is induced by a correspondence (the closure of the graph). As this rational map comes from a rational map $\PP^5\times\PP^2\dashrightarrow \PP^6$, this correspondence is the restriction of a relative correspondence.
    
    Finally, one considers the composition
      \[ V\ \xrightarrow{\Gamma}\ V\otimes H^1(E)\otimes H^1(E)\ \xrightarrow{\mu_f\otimes{\rm id}}\ W\otimes H^1(E)\ \xrightarrow{Sh\otimes{\rm id}}\ H^5(Z_b)\otimes H^1(E)\ .\]
      This composition is injective, and it is induced by a correspondence which is the restriction to $X_b\times Z_b\times E$ of a relative correspondence.
        \end{proof}

 \subsection{Splitting}
 
 For the proof of the main result, it will be useful to have splittings of the injections of subsection \ref{ssKS}.  
 
 \begin{lemma}\label{split} Let 
   \[\Gamma_{KS}\ \ \in A^5\bigl( \XX\times_B (\Z\times E)\bigr)\] 
   be a relative Kuga--Satake correspondence as in 
  theorem \ref{GIfam}. For any $b\in B$ there exists a correspondence $\Lambda_b\in A^{5}(Z_b\times E\times X_b)$ such that
    \[  H^4(X_b)_{prim}\ \xrightarrow{(\Gamma_{KS,b})_\ast}\ H^6(Z_b\times E)\ \xrightarrow{(\Lambda_b)_\ast}\ H^4(X_b)_{prim} \]
    is the identity.
    \end{lemma}
    
  \begin{proof} The varieties $X_b, Z_b$ and $E$ verify the Lefschetz standard conjecture, and hence homological and numerical equivalence coincide for all powers and products of $X_b, Z_b, E$ \cite{K0}, \cite{K}. It follows that the homological motives
    \[ h^4(X_b)\ ,\ \ h^6(Z_b\times E)\ \ \in \MM_{\rm hom}\]
    are contained in a semisimple subcategory $\MM_{\rm hom}^{\circ}\subset\MM_{\rm hom}$ (one may define $\MM_{\rm hom}^{\circ}$ as the full additive subcategory generated by motives of varieties for which the Lefschetz standard conjecture is known; it follows from \cite{J} that $\MM_{\rm hom}^{\circ}$ is semisimple).
  
  Theorem \ref{thGI}, combined with semisimplicity, now implies that
    \[ \Gamma_{KS,b}\colon\ \ h^4(X_b)\ \to\ h^6(Z_b\times E)\ \ \hbox{in}\ \MM_{\rm hom}^\circ\]
    is a split injection, i.e. there exists a correspondence $\Lambda_b$ as in lemma \ref{split}.
   \end{proof}
   
  The splitting of lemma \ref{split} can be extended to the family, in the following sense:
  
  \begin{proposition}\label{splitfam}  Let 
    \[\Gamma_{KS}\ \ \in A^5\bigl( \XX\times_B (\Z\times E)\bigr)\] 
    be a relative Kuga--Satake correspondence as in 
  theorem \ref{GIfam}. There exists a relative correspondence 
    \[\Lambda\in A^4\bigl( (\Z\times E)\times_B   \XX\bigr)\ ,\] 
    such that for any $b\in B$ we have that
  \[  H^4(X_b)_{prim}\ \xrightarrow{(\Gamma_{KS,b})_\ast}\ H^6(Z_b\times E)\ \xrightarrow{(\Lambda\vert_b)_\ast}\ H^4(X_b)_{prim} \]
    is the identity, where $\Lambda\vert_b:=\Lambda\vert_{Z_b\times E\times X_b}\in A^4(Z_b\times E\times X_b)$.
   \end{proposition}
   
  \begin{proof} This uses the idea of ``spreading out'' algebraic cycles, as advocated in \cite{V0}, \cite{V1}, \cite{Vo}. Lemma \ref{split}, plus the observation that $\ima\bigl(H^\ast(\PP^5)\to H^\ast(X_b)\bigr)$ is generated by linear subspace sections, gives a decomposition of the diagonal of $X_b$:
    \[ \Delta_{X_b}= \Lambda_b\circ \Gamma_{KS,b}+{\displaystyle \sum}_j c_j (H_b)^j\times (H_b)^{4-j}\ \ \hbox{in}\ H^8(X_b\times X_b)\ ,\]
    where $c_j\in\QQ$ and $H_b\in A^1(X_b)$ is the restriction of an ample class $H\in A^1(\PP^5)$. That is, the relative correspondences
    \[ \Delta_{\XX,prim}:= \Delta_{\XX}- \Bigl( {\displaystyle \sum}_j c_j H^j\times H^{4-j}  \times B\Bigr){}\vert_{\XX\times_B \XX}\ \ \in A^4(\XX\times_B \XX) \]
    and
    \[ \Gamma_{KS}\ \ \in A^5\bigl( \XX\times_B (\Z\times E)\bigr)\]
 have the following property: for any $b\in B$, there exists a correspondence $\Lambda_b\in A^4(Z_b\times E\times X_b)$ such that
  \[ \Delta_{\XX,prim}\vert_b = \Lambda_b\circ (\Gamma_{KS})\vert_b\ \ \in H^8(X_b\times X_b)\ .\]
  
  We now apply Voisin's argument, in the form of proposition \ref{spread} below, to finish the proof.
  \end{proof}  
  
  \begin{proposition}[Voisin \cite{V0}, \cite{V1}]\label{spread} Let $\XX$, $\YY$ and $\Z$ be families over $B$, and assume the morphisms to $B$ are smooth projective and the total spaces are smooth quasi--projective. Let
    \[ \begin{split}  \Gamma&\in\ \  A^i(\XX\times_B \Z)\ ,\\
                               \Psi&\in \ \ A^j(\XX\times_B \YY)\\
                                  \end{split}   \]
 be relative correspondences, with the property that for any $b\in B$ there exists $\Lambda_b\in A^{\ast}(Y_b\times Z_b)$ such that
             \[  \Gamma\vert_b= \Lambda_b\circ (\Psi)\vert_b\ \ \hbox{in}\ H^{2i}(X_b\times Z_b)\ .\]
      Then there exists a relative correspondence
      \[ \Lambda\ \ \in A^{\ast}(\YY\times_B \Z) \]
      with the property that for any $b\in B$
      \[ \Gamma\vert_b=(\Lambda)\vert_b \circ (\Psi)\vert_b\ \ \hbox{in}\ H^{2i}(X_b\times Z_b)\ .\]
   \end{proposition}
   
   \begin{proof} The statement is different, but this is really the same Hilbert schemes argument as \cite[Proposition 2.7]{V0}, \cite[Proposition 4.25]{Vo}. The point is that the data of 
   all the
   $(b,\Lambda_b)$ that are solutions to the splitting problem
   \[   \Gamma\vert_b= \Lambda_b\circ (\Psi)\vert_b\ \ \hbox{in}\ H^{2i}(X_b\times Z_b)\ \]
   can be encoded by a countable number of algebraic varieties $p_j\colon M_j\to B$, with universal objects $\Lambda_j\subset \YY\times_{M_j}\Z$, with the property that 
   for $m\in M_j$ and $b=p_j(m)\in B$, we have
     \[  (\Lambda_j)\vert_{m}=\Lambda_b\ \ \hbox{in}\ H^{\ast}(Y_b\times Z_b)\ .\]
     By assumption, the union of the $M_j$ dominate $B$. Since there is a countable number, one of the $M_j$ (say $M_0$) must dominate $B$. Taking hyperplane sections, we may assume $M_0\to B$ is generically finite (say of degree $d$). Projecting $\Lambda_0$ to $\YY\times_B \Z$ and dividing by $d$, we have obtained $\Lambda$ as requested.
   \end{proof}

  For ease of reference, we spell out the following restatement of proposition \ref{splitfam}: 
   
 \begin{corollary}\label{splitfam2} Let
   \[ \Delta_{\XX,prim}\ \ \in A^4(\XX\times_B \XX) \]
   be the ``corrected relative diagonal'' appearing in the proof of proposition \ref{splitfam}.
   Let 
     \[\Gamma_{KS}\ \ \in A^5\bigl( \XX\times_B (\Z\times E)\bigr)\] 
     be a relative Kuga--Satake correspondence as in 
  theorem \ref{GIfam}. There exists a relative correspondence 
    \[\Lambda\in A^4\bigl( (\Z\times E)\times_B   \XX\bigr)\ ,\] 
    such that for any $b\in B$ we have that
  \[  \Bigl( \Delta_{\XX,prim}- \Lambda\circ \Gamma_{KS}\Bigr){}\vert_{X_b\times X_b}=0\ \ \hbox{in}\ H^8(X_b\times X_b)\ .\]
   \end{corollary}  
   
\subsection{Algebraic cycles in a family}

The second key ingredient in this note is the machinery of ``spread'' as developed by Voisin \cite{V0}, \cite{V1}, \cite{Vo}, in order to deal efficiently with algebraic cycles in a family of varieties. 
This subsection contains a result by L. Fu, which is a version of ``spread'' adapted to dealing with non--complete linear systems.

\begin{proposition}[L. Fu \cite{LFu}]\label{triv} Let $\XX\to B$ be as in notation \ref{family}. Then
  \[    \lim_{\substack{\longrightarrow\\ B^\prime\subset B}}  A^4_{hom}(\XX^\prime\times_{B^\prime}\XX^\prime)=0\ ,\]
  where the direct limit is taken over the open subsets $B^\prime\subset B$. In other words, for an open $B^\prime\subset B$ and a homologically trivial cycle $a\in A^4_{hom}(\XX^\prime\times_{B^\prime}\XX^\prime)$, there is a smaller open $B^{\prime\prime}\subset B^\prime$, such that the restriction of $a$ to the base change $\XX^{\prime\prime}\times_{B^{\prime\prime}}\XX^{\prime\prime}$ is rationally trivial. 
 \end{proposition}
 
 \begin{proof} This is \cite[Proposition 4.1]{LFu}, applied to the family $\XX\to B$. 
 In the notation of \cite{LFu}, the closure $\bar{B}$ of the base $B$ can be written as
  $ \bar{B}=\PP\bigl(\oplus_{\underline{\alpha}\in \Lambda_0} \C \underline{x}^{\underline{\alpha}}\bigr)$, where
    \[ \Lambda_0:= \Bigl\{ \underline{\alpha}=(\alpha_0,\ldots,\alpha_5)\in \NN^5\ \vert\  \alpha_0+\cdots +\alpha_5=3\ ,\ \alpha_5=0\ \hbox{mod\ } 3\Bigr\}\ .\]
 This ensures that the proof of \cite[Proposition 4.1]{LFu} applies to the family $\XX\to B$.   
 
 (NB: to be sure, the statement of \cite[Proposition 4.1]{LFu} is geared towards families of cubic fourfolds having a finite order polarized automorphism that is symplectic, whereas the family $\XX\to B$ of notation \ref{family} corresponds to cubics invariant under a polarized order $3$ automorphism that is {\em non--symplectic\/}. However, the proof of
 \cite[Proposition 4.1]{LFu} only uses the description $\bar{B}=\PP\bigl(\oplus_{\underline{\alpha}\in \Lambda_j} \C \underline{x}^{\underline{\alpha}}\bigr)$, and {\em not\/} the symplectic/non--symplectic behaviour of the automorphism.)
  \end{proof}

   \begin{remark} Alternatively, a slightly different proof of proposition \ref{triv} could be given as follows. There is a natural map $\PP^5\to\PP:=\PP(1^5,3)$, where $\PP(1^5,3)$ is a weighted projective space \cite{Dol}.
   The family $\bar{\XX}\to\bar{B}$ corresponds to (hypersurfaces in $\PP^5$ that are inverse images of) the complete linear system $\PP H^0(\PP,\OO_\PP(3))$. Since the sheaf $\OO_{\PP}(3)$ is locally free and very ample \cite{Del}, the stratification argument of \cite{tod} applies to prove that
    \[ A_\ast^{hom}(\bar{\XX}\times_{\bar{B}}\XX)=0\ .\]
  Next, to pass to opens $B^\prime\subset\bar{B}$, we can use \cite[Proposition 4.3]{LFu} (which is based on the fact that ``the Chow motive of a cubic fourfold does not exceed the size  
 of Chow motives of surfaces'', to cite \cite[Section 4.2]{LFu}).  
 
 (NB: this alternative proof avoids recourse to \cite[Proposition 4.2]{LFu}, and only uses the easier \cite[Proposition 4.3]{LFu}.)
   \end{remark}

\section{Main}

\begin{theorem}\label{main} Let $X\subset\PP^5$ be a smooth cubic fourfold, defined by an equation
  \[  x_5^3+f(x_0,\ldots,x_4)=0\ ,\]
  where $f(x_0,\ldots,x_4)$ defines a smooth cubic threefold. Then $X$ has finite--dimensional motive (of abelian type).
  \end{theorem}
  
 \begin{proof} As before, let 
   \[ \XX\to B\]
   denote the family of smooth cubic fourfolds as in notation \ref{family}. We have seen (theorem \ref{GIfam})
  that there is a relative Kuga--Satake correspondence 
    \[\Gamma_{KS}\ \ \in A^5\bigl( \XX\times_B (\Z\times E)\bigr)\] 
    (where $\Z$ is a family of cubic fivefolds and $E$ is a fixed elliptic curve). We have also seen (corollary \ref{splitfam2}) there exists a   
     ``relative splitting''. That is, the relative correspondence
     \[ \DD:= \Delta_{\XX,prim}- \Lambda\circ \Gamma_{KS}\ \ \in A^4(\XX\times_B \XX)\] 
     has the property that restriction to any fibre is homologically trivial:   
     \[ \DD\vert_{X_b\times X_b}=0\ \ \hbox{in}\ H^8(X_b\times X_b)\ \ \ \hbox{for\ all\ }b\in B\ .\]

  We now proceed to make $\DD$ globally homologically trivial.
   The Leray spectral sequence argument of \cite[Lemmas 3.11 and 3.12]{V0} shows that there exists a cycle $\gamma\in A^4(\PP^5\times\PP^5)$ such that after shrinking $B$ (i.e. after replacing the parameter space $B$ by a smaller non--empty Zariski open subset $B^\prime$), one has   
          \[  \bigl(\DD-\gamma\bigr)\vert_{\XX^\prime\times_{B^\prime} \XX^\prime}  =0\ \ \hbox{in}\ H^8( \XX^\prime\times_{B^\prime} \XX^\prime)\ .\]
     In light of proposition \ref{triv}, this implies there exists a smaller non--empty Zariski open $B^{\prime\prime}\subset B^\prime$ and a rational equivalence
             \[  \bigl(\DD-\gamma\bigr)\vert_{\XX^{\prime\prime}\times_{B^{\prime\prime}} \XX^{\prime\prime}}   =0\ \ \hbox{in}\ A^4( \XX^{\prime\prime}\times_{B^{\prime\prime}} 
             \XX^{\prime\prime})\ .\]
          
      In particular, when restricting to a fibre we find that    
      \[ \bigl(  \DD-\gamma\bigr)\vert_{X_b\times X_b}=0\ \ \   \hbox{in}\ A^4(X_b\times X_b)\ \ \ \forall b\in B^{\prime\prime}\ .\]
      Now, \cite[Lemma 3.2]{Vo} implies that the same actually holds for {\em every\/} fibre over $B$, i.e.
       \[ \bigl(  \DD-\gamma\bigr)\vert_{X_b\times X_b}=0\ \ \   \hbox{in}\ A^4(X_b\times X_b)\ \ \ \forall b\in B^{}\ .\]              
            Plugging in the definition of $\DD$, this implies that for any $b\in B$, we have a rational equivalence
     \begin{equation}\label{rat}  \Delta_{X_b}= \Lambda_b\circ \Gamma_{KS,b}+R\ \ \hbox{in}\ A^4(X_b\times X_b)\ ,\end{equation}
     where $R$ is a sum of ``completely decomposed correspondences''
       \[ R={\displaystyle \sum}_i R_i={\displaystyle \sum}_i  c_i H^i\times H^{4-i}\ \ \in A^4(X_b\times X_b)\ \]
       (with $c_i\in\QQ$ and $H\in \ima\bigl(A^1(\PP^5)\to A^1(X_b)\bigr)$ an ample class). 
       
       We define a ``primitive diagonal''
       \[ \Delta^-_{X_b}:=\Delta_{X_b}+{\displaystyle \sum}_i  d_i H^i\times H^{4-i}\ \ \in A^4(X_b\times X_b)\ ,\]
       where the constants $d_i$ are such that the push--forward
       \[  (i_b\times i_b)_\ast (\Delta^-_{X_b})=0\ \ \hbox{in}\ A^6(\PP^5\times\PP^5)\]
       (here $i_b$ denotes the inclusion $X_b\to\PP^5$). Since the correspondence $R$ is the restriction of something from $\PP^5\times\PP^5$, we have that
       \[ R\circ \Delta^-_{X_b}=0\ \ \hbox{in}\ A^4(X_b\times X_b)\ .\]
       It thus follows from equality (\ref{rat}) that
       \[  \Delta^-_{X_b}= \Lambda_b\circ \Gamma_{KS,b}\circ \Delta^-_{X_b}\ \ \hbox{in}\ A^4(X_b\times X_b)\ ,\]
       i.e. the homomorphism of motives
       \[ (X_b, \Delta_{X_b}^-,0)\ \to\ h(Z_b)\otimes h(E)(-1)\ \ \hbox{in}\ \MM_{\rm rat}\]    
       has a left--inverse. This implies there also is a homomorphism
       \[ h(X_b)\ \to\ h(Z_b)\otimes h(E)(-1)\oplus \bigoplus_i \LL(m_i)\ \ \hbox{in}\ \MM_{\rm rat}\ ,\]
       exhibiting $h(X_b)$ as a direct summand of the right--hand--side. Now we note that the cubic fivefold $Z_b$
       has
       \[  A^j_{AJ}(Z_b)=0\ \ \hbox{for\ all\ }j\ \]
       (\cite{Lew}, or \cite{Ot} or \cite{HI}).
       This implies (using \cite[Theorem 4]{V2}) that the fivefold $Z_b$ has finite--dimensional motive. Since $E$ is a curve, $h(Z_b)\otimes h(E)$ is also a finite--dimensional motive, and so we have exhibited
       $h(X_b)$ as direct summand of a finite--dimensional motive.

    \end{proof}

For later use, we observe that we can also obtain a version of corollary \ref{abvar} on the level of Chow motives:

\begin{corollary}\label{KSab} Let $X$ be a smooth cubic fourfold as in theorem \ref{main}. There exist an abelian variety $A$ of dimension $g=22$, and a homomorphism
  \[  f\colon\ \ h(X)\ \to\ h^{2g-2}(A)(3-g)\oplus \bigoplus_j \LL(m_j)\ \ \hbox{in}\ \MM_{\rm rat}\ , \]
 which identifies $h(X)$ with a direct summand of the right--hand--side.
  
  (In particular, there is a correspondence $\Psi\in A^{g+1}(X\times A)$ inducing split injections
  \[ \Psi_\ast\colon\ \ A^3_{hom}(X)\ \to\ A^{g}_{(2)}(A)\ .)\]
   \end{corollary} 

\begin{proof} The proof of theorem \ref{main} gives a homomorphism
  \[  h(X)\ \to\ h^6(Z\times E)(-1)\oplus \bigoplus_i \LL(m_i)\ \ \ \hbox{in}\ \MM_{\rm rat} \]
  admitting a left--inverse, where $Z$ is a cubic fivefold.
  
 We have seen (in the proof of corollary \ref{abvar}) that there also exists a homomorphism
 \[ h(Z\times E)\ \to\ h^2(A)(2)\oplus \bigoplus_j \LL(m_j)\ \ \ \hbox{in}\ \MM_{\rm rat} \]
  admitting a left--inverse.  
  
  Combining these two, we obtain a homomorphism
  \[ h(X)\ \to\ h^2(A)(1)\oplus \bigoplus_j \LL(m_j)\ \ \ \hbox{in}\ \MM_{\rm rat} \]
  admitting a left--inverse. Composing with a Lefschetz operator on $A$, one obtains a homomorphism
     \[  f\colon\ \ h(X)\ \to\ h^{2g-2}(A)(3-g)\oplus \bigoplus_j \LL(m_j)\ \ \hbox{in}\ \MM_{\rm rat}\  \]
     that admits a left--inverse, i.e. $h(X)$ identifies with a direct summand of the right--hand--side.  
 \end{proof} 
  
 \begin{remark}\label{voisinK3} The argument used to prove theorem \ref{main} is hardly original, and I do not claim credit for this argument. Indeed, a similar use of the Kuga--Satake construction in a family appears in \cite{V8}. More precisely: Voisin proves in \cite[Theorem 0.7]{V8} that {\em if\/} the variational Hodge conjecture is true, then the Kuga--Satake construction is algebraic, and consequently a certain large family of $K3$ surfaces (obtained as sections of a vector bundle on a rationally connected variety) has finite--dimensional motive.
 
 It is also worth mentioning that an explicit Kuga--Satake construction for the $4$--dimensional subfamily of cubics of the form
   \[  x_5^3 + x_4^3 +f(x_0,\ldots,x_3)=0 \]
  already appears in \cite[Example 4.2]{V9}. This construction in \cite{V9} is mentioned by van Geemen as inspiration for his general theory of half twist \cite[Introduction]{vG2}. 
 \end{remark}

 \begin{remark} The family of cubic fourfolds $X$ of theorem \ref{main} is studied from a lattice--theoretic viewpoint in \cite[Example 6.4]{BCS}. Among other things, they prove that the natural $\ZZ/3\ZZ$ action (defined by the automorphism we denoted $\alpha_4$ in the proof of theorem \ref{GIfam} above) has the property that
   \[ \dim H^4(X)^{\ZZ/3\ZZ}=1\ ,\]
   and so
   \[ H^4(X)_{prim}\cap H^4(X)^{\ZZ/3\ZZ}=0\ .\]  
   \end{remark}

\section{Consequences}

\subsection{Bloch conjecture}

\begin{corollary}\label{corbloch} Let $X$ be a cubic fourfold as in theorem \ref{main}. Let $\Gamma\in A^4(X\times X)$ be a correspondence such that
  \[  \Gamma_\ast\colon\ \ H^{3,1}(X)\ \to\ H^{3,1}(X)\]
  is the identity. Then
  \[ \Gamma_\ast\colon\ \ A^3_{hom}(X)\ \to\ A^3_{hom}(X)\]
  is an isomorphism.
 \end{corollary}
 
 \begin{proof} As is well--known, this is a consequence of finite--dimensionality; we include a proof for completeness' sake. Using an argument involving the truth of the Hodge conjecture for $X$ and non--degeneracy of the cup--product pairing (similar to \cite[Proof of corollary 3.11]{V8} and \cite[Lemma 2.5]{Ped}, where this is done for $K3$ surfaces), the assumption implies that
 \[ \Gamma_\ast\colon\ \  H^4_{tr}(X)\ \to\ H^4_{tr}(X)\ \]
 is also the identity, where $H^4_{tr}$ denotes the orthogonal complement (under the cup--product pairing) of $N^2 H^4(X)$. It follows there is a cohomological decomposition 
 \[ \Gamma=\Delta_X+\gamma\ \ \in H^8(X\times X)\ ,\]
 where $\gamma$ is a cycle supported on $(Y\times X)\cup(X\times Y)$, for some $Y\subset X$ of codimension $2$. That is, the cycle
 \[ \Gamma-\Delta_X-\gamma\ \ \in A^4(X\times X)\]
 is homologically trivial. Using finite--dimensionality of $X$, this cycle is nilpotent. The cycle $\gamma$ does not act on $A^3_{hom}(X)=A^3_{AJ}(X)$ for dimension reasons. It follows that
   \[ (\Gamma^{\circ N})_\ast=\hbox{id}\colon\ \ A^3_{hom}(X)\ \to\ A^3_{hom}(X) \]
   for some $N\in\NN$.
 \end{proof}
 
 \begin{remark} Corollary \ref{corbloch} establishes a weak form of the Bloch conjecture \cite{B}. Recall that the Bloch conjecture (in the special case of a cubic fourfold $X$) predicts that if a correspondence acts as the identity on $H^{3,1}(X)$, then it acts as the identity on $A^3_{hom}(X)$. 
 
 There is related work of L. Fu \cite{LFu}, proving that for {\em any\/} cubic fourfold, Bloch's conjecture is true for the graph of an automorphism acting as the identity on $H^{3,1}(X)$.
 \end{remark}

\subsection{The Fano variety of lines}

\begin{corollary}\label{fano} Let $X$ be a smooth cubic fourfold as in theorem \ref{main}, and let $F(X)$ be the Fano variety of lines on $X$. Then $F(X)$ has finite--dimensional motive.
\end{corollary}

\begin{proof} This follows from the main result of \cite{moi}.
\end{proof}

\begin{remark} Corollary \ref{fano} can be extended to hyperk\"ahler fourfolds that are birational to $F(X)$. Indeed, the isomorphism of Rie\ss \cite{Rie} implies that birational hyperk\"ahler varieties have isomorphic Chow motives.
\end{remark}

\subsection{Indecomposability}

\begin{theorem}[Vial \cite{V4}] Let $M$ be a smooth projective variety of dimension $n \le 5$. Assume that $M$ has finite--dimensional motive, and that the standard Lefschetz conjecture $B(M)$ holds. Then there exists a refined Chow--K\"unneth decomposition, i.e. a set of mutually orthogonal idempotents
  \[ \Pi_{i,j}\ \ \in A^n(M\times M)\ \ \ ,\]
  such that $\Pi_{i,j}$ acts on cohomology as a projector on $\gr^j_{\wt{N}} H^i(M)$, where $\wt{N}^\ast$ is the niveau filtration of \cite{V4}. 
  \end{theorem}
  
  \begin{proof} This is a combination of \cite[Theorems 1 and 2]{V4}, since $M$ verifies conditions (*) and (**) of loc. cit.
  \end{proof}

  \begin{remark} The ``niveau filtration'' $\wt{N}^\ast$ of \cite{V4} is a variant of the geometric coniveau filtration $N^\ast$ of \cite{BO}. It is expected that there is equality $\wt{N}^\ast= N^\ast$; this is true if the standard Lefschetz conjecture is true for all smooth projective varieties \cite{V4}.
  \end{remark}

\begin{definition} Let $X$ be a cubic fourfold as in theorem \ref{main}. We define the ``transcendental motive'' $t(X)\in\MM_{\rm rat}$ as
  \[ t(X)=(X,\Pi_{4,1},0)\ \ \in \MM_{\rm rat}\ ,\]
 where the $\Pi_{i,j}$ are Vial's refined Chow--K\"unneth decomposition \cite[Theorems 1 and 2]{V4}. 
\end{definition}

\begin{remark} The fact that $t(X)$ is well--defined (i.e., independent of choices up to isomorphism) follows from \cite{V4} and \cite[Theorem 7.7.3]{KMP}. 

The motive $t(X)$ is an analogue of the ``transcendental part of the motive'' $t_2(X)$ that is defined for any (not necessarily finite--dimensional) surface in \cite{KMP}. 
Just like in the surface case, the motive $t(X)$ can actually be defined for any (not necessarily finite--dimensional) cubic fourfold, cf. \cite[(4.1)]{Ped2}.
\end{remark}

\begin{proposition}\label{indecomp} Let $X$ be a cubic fourfold as in theorem \ref{main}. The motive $t(X)$ is indecomposable, i.e. any submotive is either $0$ or equal to $t(X)$.
\end{proposition}

\begin{proof} Let $M\in\MM_{\rm rat}$ be a submotive of $t(X)$. Then
  \[  0\ \subset\ H^\ast(M)\ \subset\ H^\ast(t(X))=H^4_{tr}(X)\ ,\]
  where $H^4_{tr}(X)\subset H^4(X)$ is as in the proof of corollary \ref{corbloch}. The cup--product argument of the proof of corollary \ref{corbloch}, plus the fact that $h^{3,1}(X)=1$, implies that the Hodge structure $H^4_{tr}(X)$ is indecomposable. That is, $H^\ast(M)$ is either $0$ or all of $H^4_{tr}(X)$. In the first case, we conclude that $M=0$ (there are no finite--dimensional phantom motives). In the second case, we conclude (again using finite--dimensionality) that $M=t(X)$, since they coincide in $\MM_{\rm hom}$.
\end{proof}

\begin{corollary} Let $X$ be a cubic fourfold as in theorem \ref{main}. Suppose $G\subset\hbox{Aut}(X)$ is a finite group of finite--order automorphisms such that
  \[  g_\ast\not=\hbox{id}\colon\ \ H^{3,1}(X)\ \to\ H^{3,1}(X) \]
  for some $g\in G$. Let $Y\to X/G$ be a resolution of singularities of the quotient. Then
    \[ A^j_{hom}(Y)=0\ \ \hbox{for\ all\ }j\ .\]  
\end{corollary}

\begin{proof} We have
  \[ A^j_{hom}(Y)\cong A^j(t(X)^G)\ ,\]
  where we define
  \[  t(X)^G:= (X,\Pi_{4,1}\circ {\displaystyle \sum_{g\in G}} \Gamma_g,0)\ \ \in \MM_{\rm rat}\ .\]
  This is a submotive of $t(X)$; as such, it must be $0$ or all of $t(X)$. The second possibility can be excluded, because it would imply
    \[ H^{3,1}_{}(X)^G=H^{3,1}_{}(X)\ ,\]
  contradicting the hypothesis.
    \end{proof}

\subsection{Smash--equivalence}

\begin{definition} Let $X$ be a smooth projective variety. A cycle $a\in A^i(X)$ is called {\em smash--nilpotent\/} 
if there exists $m\in\NN$ such that
  \[ \begin{array}[c]{ccc}  a^m:= &\undermat{(m\hbox{ times})}{a\times\cdots\times a}&=0\ \ \hbox{in}\  A^{mi}(X\times\cdots\times X)_{}\ .
  \end{array}\]
  \vskip0.6cm

We will write $A^i_\otimes(X)\subset A^r(X)$ for the subgroup of smash--nilpotent cycles.
\end{definition}

\begin{conjecture}[Voevodsky \cite{Voe}]\label{voe} Let $X$ be a smooth projective variety. Then
  \[  A^i_{num}(X)\ \subset\ A^i_\otimes(X)\ \ \ \hbox{for\ all\ }i\ .\]
  \end{conjecture}

\begin{remark} It is known \cite[Th\'eor\`eme 3.33]{An} that conjecture \ref{voe} implies (and is strictly stronger than) Kimura's conjecture ``all varieties have finite--dimensional motive''.
For partial results concerning conjecture \ref{voe}, cf. \cite{KS}, \cite{Seb2}, \cite{Seb}, \cite[Theorem 3.17]{V3}, \cite{moismash}.
\end{remark}

The results of this note give some new examples where Voevodsky's conjecture is verified:

\begin{proposition}\label{corsmash} Let $Z$ be a product 
  \[ Z=X_1\times X_2\ ,\]
  where the $X_j$ are smooth cubic fourfolds as in theorem \ref{main}. Then
   \[ A^i_\otimes(Z)=A^i_{num}(Z)\ \ \hbox{for\ all\ }i\not=4\ .\]
  \end{proposition}
  
  \begin{proof} We have seen (in the proof of corollary \ref{KSab}) there exists a map of motives
   \[ h(X_j)\ \to\ h^2(A)(1)\oplus \bigoplus_{m=0}^4 h(\hbox{Sp}\,\C)(m)\ \ \ \hbox{in}\ \MM_{\rm rat} \]
  that admits a left--inverse.
  It follows there is also a map
  \[ h(Z)= h(X_1\times X_2)\ \to\ h^4(A\times A)(2)\oplus \bigoplus_{m^\prime=1}^5 h^2(A)(m^\prime)\oplus \bigoplus_{m^{\prime\prime}} h(\hbox{Sp}\,\C)(m^{\prime\prime})\  \ \ \hbox{in}\ \MM_{\rm rat}\ \]
  admitting a left--inverse.
  In particular, this implies there is a correspondence--induced injection
   \begin{equation}\label{inject} A^i_{num}(Z)\ \hookrightarrow\ A^{i-2}_{(2i-8)}(A\times A)\oplus \bigoplus_{m^\prime} (\pi_2^A)_\ast A^{i-m^\prime}(A)\ .\end{equation}
   By general properties of Beauville's splitting \cite{Beau}, we know that the term $(\pi_2^A)_\ast A^{i-m^\prime}(A)$ is $0$ unless $i-m^\prime$ is $1$ or $2$. For $i-m^\prime=1$, we have
   \[    (\pi_2^A)_\ast A^{1}(A)  = A^1_{(0)}(A) \ ,\]
   which is known to have trivial intersection with $A^1_{num}(A)$. For $i-m^\prime=2$, we have
    \[    (\pi_2^A)_\ast A^{2}(A)  = A^2_{(2)}(A) \ \xrightarrow{\cong}\ A^g_{(2)}(A)\ ,  \]
    where the isomorphism is given by K\"unnemann's hard Lefschetz result \cite{Kun}, which implies
    \[  (\pi_2^A)_\ast A^{2}(A)\ \subset\ A^2_\otimes(A)\ .\] 
    It remains to analyze the first summand of the right--hand side of (\ref{inject}). 
   For $i>6$ we have that $2i-8>i-2$ and this summand vanishes \cite{Beau}. For $i=6$, this summand is
   \[ A^4_{(4)}(A\times A)\ \xrightarrow{\cong}\ A^{2g}_{(4)}(A\times A)\ ,\]
   which proves this summand is smash--nilpotent. For $i=5$, this summand is
   \[ A^3_{(2)}(A\times A)\ \xrightarrow{\cong}\ A^{2g-1}_{(2)}(A\times A) \ ,\]
   and so this summand is again smash--nilpotent, because homologically trivial $1$--cycles on abelian varieties are smash--nilpotent \cite{Seb}. 
   
   This proves the proposition: for any $i\not=4$, we have checked that the injection (\ref{inject}) sends $A^i_{num}(Z)$ to something smash--nilpotent. The left inverse of (\ref{inject}) being given by a correspondence, this implies that any element in $A^i_{num}(Z)$ is smash--nilpotent.
    
    (NB: this proof breaks down for $i=4$, because it is not known whether 
     \[ A^2_{(0)}(A\times A)\cap A^2_{num}(A\times A) = 0\ ,\]
     which is one of Beauville's conjectures.)   
         \end{proof}

\vskip1cm
\begin{nonumberingt} The ideas developed in this note grew into being during the Strasbourg 2014---2015 groupe de travail based on the monograph \cite{Vo}. Thanks to all participants of this groupe de travail for the stimulating atmosphere. Thanks to Bert van Geemen for helpful email correspondence, and thanks to the referee for insightful remarks and suggestions that significantly improved this note.
Many thanks to Yasuyo, Kai and Len for being dedicated members of the Schiltigheim Math Research Institute.
\end{nonumberingt}


\vskip1cm
\begin{thebibliography}{dlPG99}


\bibitem{ACV} J. Achter, S. Casalaina--Martin and C. Vial, On descending cohomology geometrically, to appear in Comp. Math., arXiv:1410.5376,

\bibitem{An} Y. Andr\'e, Motifs de dimension finie (d'apr\`es S.-I. Kimura, P. O'Sullivan,...), S\'eminaire Bourbaki 2003/2004, Ast\'erisque 299 Exp. No. 929, viii, 115---145,

\bibitem{Beau} A. Beauville, Sur l'anneau de Chow d'une vari\'et\'e ab\'elienne, Math. Ann. 273 (1986), 647---651,


\bibitem{B} S. Bloch, Lectures on algebraic cycles, Duke Univ. Press Durham 1980,

\bibitem{BO} S. Bloch and A. Ogus, Gersten's conjecture and the homology of schemes, Ann. Sci. Ecole Norm. Sup. 4 (1974), 181---202,

\bibitem{BS} S. Bloch and V. Srinivas, Remarks on correspondences and algebraic cycles, American Journal of Mathematics Vol. 105, No 5 (1983), 1235---1253,

\bibitem{BCS} S. Boissi\`ere, C. Camere and A. Sarti, Classification of automorphisms on a deformation family of hyperk\"ahler fourfolds by $p$--elementary 
lattices, Kyoto Journal of Math.56 no. 3 (2016), 465---499,

\bibitem{Br} M. Brion, Log homogeneous varieties, in: Actas del XVI Coloquio Latinoamericano de Algebra, 
Revista Matem\'atica Iberoamericana, Madrid 2007,
arXiv: math/0609669,

\bibitem{CM} 
M. de Cataldo and L. Migliorini, The Chow groups and the motive of the Hilbert scheme of points on a
surface, Journal of Algebra 251 no. 2 (2002), 824---848,

\bibitem{Col} A. Collino, The Abel--Jacobi isomorphism for the cubic fivefold, Pacific Journal of Mathematics Vol. 122 No. 1 (1986), 43---55,


\bibitem{D} P. Deligne, La conjecture de Weil pour les surfaces $K3$, Invent. Math. 15 (1972), 206---226,

\bibitem{Del} C. Delorme, Espaces projectifs anisotropes, Bull. Soc. Math. France 103 (1975), 203---223,

\bibitem{DM} C. Deninger and J. Murre, Motivic decomposition of abelian schemes and the Fourier transform. J. reine u.
angew. Math. 422 (1991), 201---219,

\bibitem{Dol} I. Dolgachev, Weighted projective varieties, in: Group actions and vector fields, Vancouver 1981, Springer Lecture Notes in Mathematics 956, Springer Berlin Heidelberg New York 1982,

\bibitem{LFu} L. Fu, On the action of symplectic automorphisms on the $CH_0$--groups of some hyper-K\"ahler fourfolds, Math. Z. 280 (2015), 307---334,

\bibitem{F} W. Fulton, Intersection theory, Springer--Verlag Ergebnisse der Mathematik, Berlin Heidelberg New York Tokyo 1984,

\bibitem{vG} B. van Geemen, Kuga--Satake varieties and the Hodge conjecture, in:
The Arithmetic and Geometry of Algebraic Cycles, Banff 1998 (B. Gordon et alii, eds.), Kluwer Dordrecht 2000,

\bibitem{vG2} B. van Geemen, Half twists of Hodge structures of CM--type, J. Math. Soc. Japan Vol. 53 No. 4 (2001), 813---833,

\bibitem{GI} B. van Geemen and E. Izadi, Half twists and the cohomology of hypersurfaces, Math. Z. 242 (2002), 279---301,


\bibitem{GP} V. Guletski\u{\i} and C. Pedrini, The Chow motive of the Godeaux surface, in:
Algebraic Geometry, a volume in memory of Paolo Francia (M.C. Beltrametti,
F. Catanese, C. Ciliberto, A. Lanteri and C. Pedrini, editors),
Walter de Gruyter, Berlin New York, 2002,

\bibitem{HI} A. Hirschowitz and J. Iyer, Hilbert schemes of fat r--planes and the triviality of Chow groups of complete intersections. In: Vector
bundles and complex geometry, Contemp. Math. 522, Amer. Math. Soc., Providence (2010),

\bibitem{Iv} F. Ivorra, Finite dimensional motives and applications (following S.-I. Kimura, P. O'Sullivan and others), in:
        Autour des motifs, Asian-French summer school on algebraic geometry and number theory,
       Volume III, Panoramas et synth\`eses, Soci\'et\'e math\'ematique de France 2011,
    
\bibitem{Iy} J. Iyer, Murre's conjectures and explicit Chow--K\"unneth projectors for varieties with a nef tangent bundle, Transactions of the Amer. Math. Soc. 361 (2008), 1667---1681,

\bibitem{Iy2} J. Iyer, Absolute Chow--K\"unneth decomposition for rational homogeneous bundles and for log homogeneous varieties, Michigan Math. Journal
 Vol.60, 1 (2011), 79---91,
          

\bibitem{J} U. Jannsen, 
Motives, numerical equivalence, and semi-simplicity, Invent. Math. 107(3) (1992), 447---452, 



\bibitem{J4} U. Jannsen, On finite--dimensional motives and Murre's conjecture, in: Algebraic cycles and motives (J. Nagel and C. Peters, eds.), Cambridge University Press, Cambridge 2007,

\bibitem{KMP} B. Kahn, J. Murre and C. Pedrini, On the transcendental part of the motive of a surface, in: Algebraic cycles and motives (J. Nagel and C. Peters, eds.), Cambridge University Press, Cambridge 2007,

\bibitem{KS} B. Kahn and R. Sebastian, Smash--nilpotent cycles on abelian 3--folds, Math. Res. Letters 16 (2009), 1007---1010,

\bibitem{Kim} S. Kimura, Chow groups are finite dimensional, in some sense,
Math. Ann. 331 (2005), 173---201,

\bibitem{K0} S. Kleiman, Algebraic cycles and the Weil conjectures, in: Dix expos\'es sur la cohomologie des sch\'emas, North--Holland Amsterdam, 1968, 359---386, 

\bibitem{K} S. Kleiman, The standard conjectures, in: Motives (U. Jannsen et alii, eds.), Proceedings of Symposia in Pure Mathematics Vol. 55 (1994), Part 1, 

\bibitem{Kun} K. K\"unnemann, A Lefschetz decomposition for Chow motives of abelian schemes, Inv. Math. 113 (1993), 85---102,







\bibitem{tod} R. Laterveer, Algebraic cycles and Todorov surfaces, to appear in Kyoto Journal of Mathematics, arXiv:1609.09629,

\bibitem{moi} R. Laterveer, A remark on the motive of the Fano variety of lines of a cubic, to appear in Ann. Math. Qu\'ebec, arXiv:1611.08818,

\bibitem{moismash} R. Laterveer, Some new examples of smash--nilpotent algebraic cycles, to appear in Glasgow Math. Journal, arXiv:1609.08799,

\bibitem{Lew} J. Lewis, Cylinder homomorphisms and Chow groups, Math. Nachr. 160 (1993), 205---221,


\bibitem{Mur} J. Murre, On a conjectural filtration on the Chow groups of an algebraic variety, parts I and II, Indag. Math. 4 (1993), 177---201,

\bibitem{MNP} J. Murre, J. Nagel and C. Peters, Lectures on the theory of pure motives, Amer. Math. Soc. University Lecture Series 61, Providence 2013,

\bibitem{Ot} A. Otwinowska, Remarques sur les groupes de Chow des hypersurfaces de petit degr\'e, C. R.
Acad. Sci. Paris S\'erie I Math. 329 (1999), no. 1, 51---56,

\bibitem{Par} K. Paranjape, Abelian varieties associated to certain K3 surfaces, Comp. Math. 68 (1988), 11---22,

\bibitem{P} C. Pedrini, On the finite dimensionality of a $K3$ surface, Manuscripta Mathematica 138 (2012), 59---72,

\bibitem{Ped} C. Pedrini, Bloch's conjecture and valences of correspondences for $K3$ surfaces, arXiv:1510.05832v1,

\bibitem{Ped2} C. Pedrini, On the rationality and the finite dimensionality of a cubic fourfold, arXiv:1701.05743,


\bibitem{PW} C. Pedrini and C. Weibel, Some surfaces of general type for which Bloch's conjecture holds, to appear in: Period Domains, Algebraic Cycles, and Arithmetic, Cambridge Univ. Press, 2015,



\bibitem{Rie} U. Rie\ss, On the Chow ring of birational irreducible symplectic varieties, Manuscripta Math. 145 (2014), 473---501,


\bibitem{Sc} T. Scholl, Classical motives, in: Motives (U. Jannsen et alii, eds.), Proceedings of Symposia in Pure Mathematics Vol. 55 (1994), Part 1, 

\bibitem{Seb} R. Sebastian, Smash nilpotent cycles on varieties dominated by products of curves, Comp. Math. 149 (2013), 1511---1518,

\bibitem{Seb2} R. Sebastian, Examples of smash nilpotent cycles on rationally connected varieties, Journal of Algebra 438 (2015), 119---129,

\bibitem{Shi} T. Shioda, The Hodge conjecture for Fermat varieties, Math. Ann. 245 (1979), 175---184,

\bibitem{V} C. Vial, Algebraic cycles and fibrations, Documenta Math. 18 (2013), 1521---1553,

\bibitem{V2} C. Vial, Projectors on the intermediate algebraic Jacobians, New York J. Math. 19 (2013), 793---822,

\bibitem{V3} C. Vial, Remarks on motives of abelian type, to appear in Tohoku Math. J.,

\bibitem{V4} C. Vial, Niveau and coniveau filtrations on cohomology groups and Chow groups, Proceedings of the LMS 106(2) (2013), 410---444,

\bibitem{V5} C. Vial, Chow--K\"unneth decomposition for $3$-- and $4$--folds fibred by varieties with trivial Chow group of zero--cycles, J. Alg. Geom. 24 (2015), 51---80,

\bibitem{Voe} V. Voevodsky, A nilpotence theorem for cycles algebraically equivalent to zero, Internat. Math. Research Notices 4 (1995), 187---198,

\bibitem{V9} C. Voisin, Remarks on zero--cycles of self--products of varieties, in: Moduli of vector bundles, Proceedings of the Taniguchi Congress  (M. Maruyama,  ed.), Marcel Dekker New York Basel Hong Kong 1994,





\bibitem{V0} C. Voisin, The generalized Hodge and Bloch conjectures are equivalent for general complete intersections, Ann. Sci. Ecole Norm. Sup. 46, fascicule 3 (2013), 449---475,

\bibitem{V8} C. Voisin, Bloch's conjecture for Catanese and Barlow surfaces, J. Differential Geometry 97 (2014), 149---175,

\bibitem{Vo} C. Voisin, Chow Rings, Decomposition of the Diagonal, and the Topology of Families, Princeton University Press, Princeton and Oxford, 2014,

\bibitem{V1} C. Voisin, The generalized Hodge and Bloch conjectures are equivalent for general complete intersections, II, J. Math. Sci. Univ. Tokyo  22 (2015), 491---517,

\bibitem{Xu} Z. Xu, Algebraic cycles on a generalized Kummer variety, arXiv:1506.04297v1.

\end{thebibliography}
\end{document}